\documentclass[reqno]{amsart}
\usepackage{amssymb}
\newtheorem{theorem}{Theorem}[section] 
\newtheorem{lemma}[theorem]{Lemma}     
\newtheorem{corollary}[theorem]{Corollary}

\newtheorem{remark}{Remark}

\newtheorem{example}{Example}

\newcommand{\T}{\mathbb T}
\newcommand{\C}{{\mathbb C}}
\newcommand{\D}{{\mathbb D}}
\newcommand{\cE}{{\mathcal E}}

\newcommand{\PP}{{\mathbb P}}
\newcommand{\cS}{{\mathcal S}}

\newcommand{\too}{\widehat{\to}}

\title[The boundary analog]
{The boundary analog of the Carath\'eodory-Schur interpolation problem} 

\author{Vladimir Bolotnikov}
\address{Department of Mathematics,
The College of William and Mary,
Williamsburg VA 23187-8795, USA}
\email{vladi@math.wm.edu}

\begin{document}

\begin{abstract}
Characterization of Schur-class functions (analytic and bounded by one 
in modulus on the open unit disk) in terms of their Taylor coefficients at 
the origin is due to I. Schur. We present a boundary analog 
of this result: necessary and sufficient conditions are given for the 
existence of a Schur-class function with the prescribed  nontangential 
boundary expansion 
$f(z)=s_0+s_1(z-t_0)+\ldots+s_N(z-t_0)^N+o(|z-t_0|^N)$ at a given
point $t_0$ on the unit circle.

\end{abstract}

\maketitle

\section{Introduction}
\setcounter{equation}{0}

Let $\cS$ denote the Schur class of analytic functions mapping 
the open unit disk $\D$ into its closure (i.e., the closed unit ball of 
$H^\infty$). Characterization of Schur class functions 
in terms of their Taylor coefficients goes back to I. Schur \cite{Schur} 
(and to C. Carath\'eodory \cite{cara0} for a related class of 
functions). 
\begin{theorem}
There is a function
$\; f(z)=s_0+s_1z+\ldots+s_{n-1}z^{n-1}+\ldots\in\cS\;$
if and only if the lower triangular Toeplitz matrix
${\mathbb U}^{\bf s}_{n}$ (see formula (\ref{1.5}) below) is a
contraction, i.e., if and only if the matrix
$\PP=I_n-{\mathbb U}^{\bf s}_{n}{\mathbb U}^{{\bf s}*}_{n}$ is positive
semidefinite.
\label{T:1.0}
\end{theorem}
By a conformal change in variable, a similar result is established for
an arbitrary point $\zeta\in\D$ at which the Taylor coefficients are
prescribed: there exists a function $f\in\cS$ of the form
\begin{equation}
f(z)=s_0+s_1(z-\zeta)+\ldots+s_{n-1}(z-\zeta)^{n-1}+\ldots
\label{1.1}
\end{equation}
if and only if a certain matrix $\PP$ (explicitly constructed in terms
of $\zeta$ and $s_0,\ldots,s_{n-1}$) is positive semidefinite.
Furthermore, if $\PP$ is positive definite, then there are infinitely many
functions $f\in\cS$ of the form (\ref{1.1}). If $\PP\ge 0$ is
singular, then there is a unique $f\in\cS$ of the form (\ref{1.1})
and this unique function is a finite Blaschke product of degree equal to
the rank of $\PP$.

\smallskip

In this paper, we examine a similar question in the ``boundary''  setting 
where Taylor expansion (\ref{1.1}) at $\zeta\in\D$ is 
replaced by the 
asymptotic expansion at some point $t_0$ on the unit circle $\T$.

\medskip

{\bf Question}: {\em Given a point $t_0\in\T$ and given 
numbers $s_0,\ldots,s_N\in\C$, does there exist a  
function $f\in\cS$ which admits the asymptotic expansion
\begin{equation}
f(z)=s_0+s_1(z-t_0)+\ldots+s_N(z-t_0)^N+o(|z-t_0|^N)
\label{1.2}
\end{equation}
as $z$ tends to $t_0$ nontangentially?}

\medskip 

The complete answer to this question is given in Theorem \ref{T:1.2} below
which is the main result of the paper. The necessary and sufficient conditions 
for the existence of a function $f\in\cS$ subject to \eqref{1.2} are given 
in terms of a certain positive semidefinite matrix (as in the classical ``interior'' case)
constructed explicitly in terms of the data set and (in contrast to the classical case) 
of two additional numbers also constructed from $\{t_0,s_0,\ldots,s_N\}$. This theorem
also list all the cases where the uniqueness occurs; as in the classical case, the unique 
function $f$ satisfying \eqref{1.2} is always a finite Blaschke product. To conclude this 
introduction we mention two questions beyond the one considered here. The first is to
find necessary and sufficient conditions for the existence of a Schur-class function 
with prescribed asymptotics at several (maybe countably many) boundary points; the question becomes 
even more intriguing if some of the prescribed asymptotics are infinite. Another question is 
to describe all Schur-class functions with the prescribed boundary asymptotics
(up to this point, such a description is known only for the case where 
$s_0,\ldots,s_{N}\}$ satisfy some very special symmetry relations (see 
e.g., \cite{ball8, BGR, ballhelton1, bkiwota, boldym1, Kov}). We hope to address these issues
on separate occasions. 

\smallskip

The paper is organized as follows: in Section 2 we present some needed 
preliminaries and formulate the main result. Its proof is given in the two last sections.

\section{Preliminaries and the formulation of the main result}
\setcounter{equation}{0}

In what follows, we will write $z\too t_0$ if a point $z$ approaches 
a boundary point $t_0\in\T$ nontangentially
and we will write $z\to t_0$ if $z$ approaches $t_0$ unrestrictedly in
$\D$.  Observe that asymptotic equality (\ref{1.2}) is equivalent to
the existence of the following boundary limits  $f_j(t_0)$  and
equalities
\begin{equation}
f_j(t_0):=\lim_{z\too t_0}\frac{f^{(j)}(z)}{j!}=s_j\quad\mbox{for} \; \;
j=0,\ldots,N.
\label{1.3}
\end{equation}
Clearly, if $f$ is analytic at $t_0\in\T$, then $f_j(t_0)$ is 
the $j$-th Taylor  coefficient  of $f$ at $t_0$. We will denote by ${\bf 
BP}_N$ the interpolation problem  which consists of finding a function 
$f\in\cS$ satisfying boundary interpolation conditions (\ref{1.3}). 
\begin{lemma}
Given  $t_0\in\T$ and $s_0,\ldots,s_N\in\C$, condition $|s_0|\le 1$ is
necessary and condition $|s_0|<1$ is sufficient for the problem ${\bf
BP}_N$ to have a solution.
\label{L:1.1}
\end{lemma}
The necessity of condition $|s_0|\le 1$ follows from the very definition
of the class $\cS$. On the other hand, if $|s_0|<1$, then there are 
infinitely many functions $f\in\cS$ satisfying (\ref{1.3}); see
Theorem 1.2 in \cite{bol1} for the proof. Skipping the trivial case $N=0$ 
(where condition $|s_0|\le 
1$ is necessary and sufficient for the problem {\bf BP}$_0$ to have a solution
and in fact, infinitely many solutions), we review the case $N=1$; a short direct proof 
based on the Carath\'eodory-Julia theorem can be found in \cite{bCR}.
\begin{theorem}
Given $s_0, \, s_1\in\C$, there exists a function $f\in\cS$ such that
\begin{equation}
f(z)=s_0+s_1(z-t_0)+o(|z-t_0|)\quad\mbox{as}\quad z\too t_0
\label{1.12}
\end{equation}
if and only if either $\quad (1) \; \; |s_0|<1\quad$ or
$\quad (2) \; \; |s_0|=1 \quad
\mbox{and} \quad t_0s_1\overline{s}_0\ge 0$.
Such a function is unique and is equal identically to $s_0$
if and only if $|s_0|=1$ and $s_1=0$.
\label{T:1.3a}
\end{theorem}
Due to Theorem \ref{T:1.3a}, we may focus in what follows on the case $N\ge 2$.
Moreover, due to Lemma \ref{L:1.1} it suffices to assume that  $|s_0|=1$ and to characterize all 
tuples $\{s_1,\ldots,s_N\}$ for which 
the problem {\bf BP$_N$} has a solution under the latter assumption. 
To  present the result, we first introduce some needed definitions. In what 
follows,  $\cS^{(n)}(t_0)$ will stand for the class of Schur  
functions satisfying a Carath\'eodory-Julia type condition:
\begin{equation}
f\in \cS^{(n)}(t_0)\quad {\buildrel\rm def\over \Longleftrightarrow}\quad
f\in\cS\quad\& \quad
\liminf_{z\to t_0}\frac{\partial^{2n-2}}{\partial
z^{n-1}\partial\bar{z}^{n-1}}
 \, \frac{1-|f(z)|^2}{1-|z|^2}<\infty.
\label{1.4}   
\end{equation}
We will identify $\cS^{(0)}(t_0)$ with $\cS$. The higher order Carath\'eodory-Julia condition
(\ref{1.4}) was introduced in \cite{bkjfa} and studied later in \cite{bknach} and \cite{bkanal}.
This condition can be equivalently reformulated in terms of 
the de Branges-Rovnyak space ${\mathcal H}(f)$ (we refer to \cite{dbr2} for the definition) 
associated 
with the function $f\in\cS$ as follows: {\em a Schur-class function $f$ belongs to $\cS^{(n)}(t_0)$
if and only if for every $f\in{\mathcal H}(f)$, the boundary limits $f_j(t_0)$ exist for 
$j=0,\ldots,n-1$}.
As was shown in \cite{fm} (and earlier in \cite{acl} for inner functions), the latter de 
Branges-Rovnyak space property (and therefore, the membership in $\cS^{(n)}(t_0)$) is equivalent to 
relation
$$
\sum_k\frac{1-|a_k|^2}{|t_0-a_k|^{2n+2}}+
    \int_0^{2\pi}\frac{d\mu(\theta)}{|t_0-e^{i\theta}|^{2n+2}}<\infty,
$$
where the numbers $a_k$ come from the Blaschke product of the inner-outer factorization of $f$:
$$
f(z)=\prod_k\frac{\bar{a}_k}{a_k}\cdot\frac{z-a_k}{1-z\bar{a}_k}\cdot
\exp\left\{-\int_0^{2\pi}\frac{e^{i\theta}+z}{e^{i\theta}-z}d
\mu(\theta)\right\}.
$$
Several other equivalent characterizations of the class $\cS^{(n)}(t_0)$
will be recalled in Theorem \ref{T:2.1} below. Given a tuple ${\bf s}=\{s_0,s_1,\ldots,s_N\}$, 
we define the lower triangular Toeplitz matrix ${\mathbb U}^{\bf s}_{n}$ 
and the Hankel matrix ${\mathbb H}^{\bf s}_{n}$  by
\begin{equation}
{\mathbb U}^{\bf s}_{n}=\left[\begin{array}{cccc}s_{0} & 0 & \ldots & 0 
\\ s_{1}& s_{0} & \ddots & \vdots \\ \vdots& \ddots & \ddots & 0
\\ s_{n-1}&  \ldots & s_{1} & s_{0}
\end{array}\right],\quad {\mathbb H}^{\bf 
s}_{n}=\left[\begin{array}{ccccc}
 s_{1} &  s_{2} & \ldots & s_{n} \\
 s_{2} &  s_{3} & \ldots & s_{n+1}\\ 
\vdots & \vdots & & \vdots \\ s_{n} & s_{n+1}&\ldots &
s_{2n-1}\end{array}\right]
\label{1.5}   
\end{equation}
for every appropriate integer $n\ge 1$ (i.e., for every $n\le N+1$ in the 
first  formula and for every $n\le (N+1)/2$ in the second). Given a point 
$t_0\in\T$,
we introduce the 
upper triangular matrix
\begin{equation}
{\bf \Psi}_n(t_0)=\left[\begin{array}{cccl} t_0 & -t_0^2&
\cdots& (-1)^{n-1}{\scriptsize\left(\begin{array}{c} n-1 \\ 0
\end{array}\right)}t_0^{n}\\ 0 & -t_0^3 & \cdots &
(-1)^{n-1}{\scriptsize\left(\begin{array}{c} n-1 \\ 1   
\end{array}\right)}t_0^{n+1}\\ 
\vdots&  &\ddots &\vdots \\ 0 & \cdots & 0 & 
(-1)^{n-1}{\scriptsize\left(\begin{array}{c}
n-1 \\ n-1 \end{array}\right)}t_0^{2n-1}\end{array}\right]
\label{1.55}
\end{equation}
with the entries
\begin{equation}
\Psi_{j\ell}=\left\{\begin{array}{ccl}
0, & \mbox{if} & j>\ell, \\ 
(-1)^{\ell-1}{\scriptsize
\left(\begin{array}{c} \ell-1 \\ j-1
\end{array}\right)}t_0^{\ell+j-1}, & \mbox{if} & j\leq\ell,
\end{array}\right.\quad (j,\ell=1,\ldots,n),
\label{1.6}
\end{equation}
and finally, for every $n\le (N+1)/2$, we introduce the structured matrix
\begin{equation}
\PP^{\bf s}_n=\left[p^{\bf s}_{ij}\right]_{i,j=1}^n={\mathbb H}^{\bf 
s}_{n}{\bf \Psi}_n(t_0)
{\mathbb U}^{{\bf s} *}_{n}
\label{1.7}  
\end{equation}
with the entries (as it follows from (\ref{1.5})--(\ref{1.7}))
\begin{equation}
p^{\bf s}_{ij}=\sum_{r=1}^{j}\left(\sum_{\ell=1}^r
s_{i+\ell-1}\Psi_{\ell r}\right)\overline{s}_{j-r}.
\label{1.9}
\end{equation}
Although the matrix $\PP^{\bf s}_n$ depends on $t_0$, we drop this 
dependence from notation. However, in the case that the  parameters $s_j$ in 
(\ref{1.7}) are equal to the angular boundary limits $f_j(t_0)$ (see 
definition (\ref{1.3})) for some analytic function $f$, then we will write 
$\PP^{f}_n(t_0)$ rather than $\PP^{\bf s}_n$:
\begin{equation}
\PP^{f}_n(t_0)=\left[\begin{array}{ccc}
f_1(t_0) & \ldots & f_{n}(t_0) \\
\vdots & & \vdots \\ f_{n}(t_0) & \ldots & f_{2n-1}(t_0)
\end{array}\right]{\bf \Psi}_{n}(t_0)\left[\begin{array}{ccc}
\overline{f_0(t_0)} & \ldots & \overline{f_{n-1}(t_0)}\\
& \ddots &\vdots \\ 0 && \overline{f_0(t_0)}\end{array}\right].
\label{1.8}
\end{equation}
Due to the upper triangular structure of the factors ${\bf \Psi}_n(t_0)$
and ${\mathbb U}^{{\bf s} *}_{n}$ in (\ref{1.7}), it follows that
$\PP^{\bf s}_{k}$ is the principal submatrix of $\PP^{\bf s}_n$ for every
$k<n$. We also observe that formula (\ref{1.9})
defines the numbers $p^{\bf s}_{ij}$ in terms of ${\bf 
s}=\{s_0,\ldots,s_{N}\}$ for every pair of indices $(i,j)$ subject to
$i+j\le N+1$. In particular, if $n\le N/2$, one can define via this 
formula the column
\begin{equation}
B_n:=\left[\begin{array}{c}p^{\bf s}_{1,n+1} \\ \vdots \\
p^{\bf s}_{n,n+1}\end{array}\right]=\left[\begin{array}{ccccc}
 s_{1} &  s_{2} & \ldots & s_{n+1} \\
\vdots & \vdots & & \vdots \\ s_{n} & s_{n+1}&\ldots &
s_{2n}\end{array}\right]{\bf
\Psi}_{n+1}(t_0)\left[\begin{array}{c}\overline{s}_{n} \\
 \vdots \\ \overline{s}_{0}\end{array}\right],
\label{1.10}  
\end{equation}
where the second equality follows from representation of type (\ref{1.7})
for the matrix $\PP^{\bf s}_{n+1}$ which is determined from 
${\bf s}=\{s_0,\ldots,s_N\}$ completely (if $n<N/2$) or except for the 
entry $p^{\bf s}_{n+1,n+1}$ (if $N=2n$). 

\smallskip

The next theorem is the main result of the paper; it gives necessary and sufficient 
conditions for the problem {\bf BP}$_N$ to have a solution and also for 
this solution to be unique. 
\begin{theorem}
Let $t_0\in\T$ and ${\bf s}=\{s_0,s_1,\ldots,s_N\}$ ($N\ge 2$) be given. In case the matrix 
$\PP^{\bf s}_k$ is positive semidefinite for some $k\ge 0$, we let 
$n$ ($0\le n\le (N+1)/2$) to be the greatest integer such that $\PP^{\bf s}_n\ge 0$. 
In case $n\le N/2$, let $p^{\bf s}_{n+1,n}$ and 
$p^{\bf s}_{n,n+1}$ be defined by (\ref{1.9}) and let $B_n$ be as in 
(\ref{1.10}). Then
\begin{enumerate}
\item The problem {\bf BP$_N$} has a unique solution if and only if 
$|s_0|=1$, $\PP^{\bf s}_n$ is singular and either 
\begin{enumerate}
\item $n=(N+1)/2\quad$ and $\quad{\rm rank} \, \PP^{\bf s}_n={\rm rank} \, 
\PP^{\bf s}_{n-1}\quad$ or
\item $n=N/2$, 
\begin{equation}
p^{\bf s}_{n+1,n}=\overline{p}^{\bf s}_{n,n+1}\quad\mbox{and}\quad
{\rm rank} \, \PP^{\bf s}_n={\rm rank} \, \left[\PP^{\bf s}_n \; \;
B_n\right].
\label{1.11}
\end{equation}
\end{enumerate}
The unique solution is a finite Blaschke product of degree equal 
${\rm rank} \, \PP^{\bf s}_n$.
\item  The problem {\bf BP}$_N$ has infinitely many solutions 
if and only if either 
\begin{enumerate}
\item $|s_0|<1$ or 
\item $|s_0|=1$, $\PP^{\bf s}_n>0$ and one of the following holds:
\begin{enumerate}
\item $n=(N+1)/2$;
\item $n=N/2\; \; $ and  $\; \; t_0\cdot\left(p^{\bf 
s}_{n+1,n}-\overline{p}^{\bf s}_{n,n+1}\right)\ge 0$; 
\item $0<n<N/2\; \;$ and 
$\; \; t_0\cdot\left(p^{\bf s}_{n+1,n}-\overline{p}^{\bf 
s}_{n,n+1}\right)> 0$.
\end{enumerate}
In any of these three cases, every solution of the problem belongs to 
$\cS^{(n)}(t_0)$.
\end{enumerate}
\item Otherwise the problem has no solutions.
\end{enumerate}
\label{T:1.2}
\end{theorem}
Part (1) in Theorem \ref{T:1.2} can be formulated in the
following more unified way (see Corollary \ref{C:2.6} below for the 
proof):
\begin{lemma}
The uniqueness occurs if and only if the matrix
$\PP_n^{\bf s}$ of the maximal possible
size (i.e., with $n=\left[\frac{N+1}{2}\right]$) is positive semidefinite
(and singular) and admits a positive semidefinite extension 
$\PP^{\bf s}_{n+1}$
for an appropriate choice of $s_{2n+1}$ (in case $N=2n$) or of
$s_{2n+1}$ and $s_{2n}$ (in case $N=2n-1$).
\label{P:1.6}
\end{lemma}
Additional symmetry and rank conditions in part (1) of Theorem \ref{T:1.2}
guarantee that the above extension exists. Observe that the  
$n\times(n+1)$ matrix $\left[\PP^{\bf s}_n \; \; B_n\right]$ in 
(\ref{1.11}) is formed by the $n$ top rows of the matrix $\PP_{n+1}^{\bf 
s}$ which are completely specified by ${\bf s}=\{s_0,\ldots,s_N\}$ 
whenever $n\le N/2$.

\smallskip

If $N=1$ or $N=2$, the integer $n$ (defined as 
in Theorem \ref{T:1.2}) is at most one and it follows from formula
(\ref{1.7}) that $\PP^{\bf s}_1=p^{\bf s}_{11}={\mathbb H}_{1}^{\bf s}{\bf 
\Psi}_1(t_0){\mathbb U}^{ *}_{1}=s_1t_0\overline{s}_0$. Furthermore,
for $N=2$, formula (\ref{1.9}) gives
$$
p^{\bf s}_{21}=t_0s_2\overline{s}_0\quad\mbox{and}\quad
p^{\bf s}_{12}=|s_1|^2t_0-s_1\overline{s}_0t_0^2-s_2\overline{s}_0t_0^3.
$$
Letting $N=1$ in Theorem \ref{T:1.2} leads us to Theorem \ref{T:1.3a}, while 
letting $N=1$ gives the following result: {\em given $s_0,s_1,s_2\in\C$, 
there exists a function $f\in\cS$ such that
\begin{equation}
f(z)=s_0+s_1(z-t_0)+s_2(z-t_0)^2+o(|z-t_0|^2)\quad\mbox{as}\quad z\too t_0,
\label{1.13}
\end{equation}
if and only if either $|s_0|<1$ or
\begin{equation}
|s_0|=1, \quad s_1t_0\overline{s}_0\ge 0\quad\mbox{and}\quad
2{\rm Re} \, (t_0^2\overline{s}_0s_2)\ge |s_1|^2-t_0\overline{s}_0s_1.
\label{1.14} 
\end{equation}}
The uniqueness occurs if and only if $|s_0|=1$ and $s_1=s_2=0$ and 
he unique function of the required form is equal to $s_0$ identically.

\medskip

In general, the algorithm determining whether or not there exists a 
Schur-class function with prescribed boundary derivatives can be designed 
as follows. If $|s_0|\neq 1$, then the definitive answer comes up.
If $|s_0|=1$, we do not have to check positivity of all the matrices
$\PP^{\bf s}_{k}$ for $k=1,2,\ldots$ to find the greatest integer $n$ such 
that $\PP^{\bf s}_{n}\ge 0$. It suffices to get the greatest $n$ such that
$\PP^{\bf s}_{n}$ is Hermitian. If this Hermitian $\PP_{n}^{\bf s}$ is not
positive semidefinite, then the problem {\bf BP}$_N$ has no
solutions (see Remark \ref{R:2.3} below). If $\PP^{\bf s}_{n}$ is positive 
semidefinite (singular), then we check one of the two possibilities 
indicated in part (1) of Theorem \ref{T:1.2} depending on the parity of 
$N$. If $\PP^{\bf s}_{n}>0$, then we verify exactly one of the three 
possibilities in part (2(b)). We illustrate this strategy by a numerical 
example.
\begin{example} 
{\rm Let $N=3$, $t_0=1$, $s_0=s_1=1$ and $s_2=s_3=0$. Then formula   
(\ref{1.7}) gives $\PP_{1}=1$ and
$\PP_2={\scriptsize\left[\begin{array}{cc} 1 &
0 \\ 0 & 0\end{array}\right]}$. Thus, the greatest $n\le (N+1)/2=2$ such
that $\PP_n$ is Hermitian, is $n=2$. Since $\PP_2$ is positive
semidefinite (singular) and since ${\rm rank} \, \PP_2={\rm rank} \,  
\PP_1=1$, it follows from part (i(a)) in Theorem \ref{T:2.1} that there is
a unique function $f\in\cS$ such that $f(z)=1+(z-1)+o((z-1)^3)$
as $z\to t_0$. This unique function is clearly $f(z)\equiv z$ which thus
gives yet another proof of Theorem 2.1 in \cite{bkr}:} If $f\in\cS$ and
if $f(z)=z+o((z-1)^3)$ as $z\to 1$, then $f(z)\equiv z$.
\label{E:2.3}
\end{example}
Inn Section 3 we consider the case when the matrix $\PP^{\bf s}_{n}$ chosen as in Theorem 
\ref{T:1.2} is singular. The nondegenerate case is handled in Section 4 at the end of which we 
summarize all possible cases completing the proof of Theorem \ref{T:1.2}.

\section{The determinate case}
\setcounter{equation}{0}

In this section we will consider the case when for some $n\le (N+1)/2$,
the matrix $\PP^{\bf s}_n$ constructed from the data set via formula 
(\ref{1.7}) is positive semidefinite and singular. It is well known that 
for any Schur-class function $f$, the Schwarz-Pick matrix 
$$
{\bf P}^f_{n}(z):=\left[\frac{1}{i!j!} \,
\frac{\partial^{i+j}}{\partial z^{i}\partial\bar{z}^{j}} \,
\frac{1-|f(z)|^2}{1-|z|^2}\right]_{i,j=0}^{n-1}
$$
is positive semidefinite for every $n\ge 1$ and $z\in\D$; in fact it is 
positive definite unless $f$ is a finite Blaschke product in which case 
${\rm rank} ({\bf P}^f_{n}(z))=\min\{n, \deg f\}$. Given a point 
$t_0\in\T$, the {\em boundary Schwarz-Pick matrix} is defined by
\begin{equation}
{\bf P}^f_{n}(t_0):=\lim_{z\too t_0}{\bf P}^f_{n}(z),
\label{2.1}
\end{equation}
provided the nontangential limit in $(\ref{2.1})$ exists. Thus,
once the boundary Schwarz-Pick matrix ${\bf P}^f_{n}(t_0)$ exists, it is
positive semidefinite. It is readily seen from  
definition (\ref{1.4}) that the membership $f\in\cS^{(n)}(t_0)$ is 
necessary for the limit (\ref{2.1}) to exist (it is necessary for 
the nontangential convergence of the rightmost diagonal entry in 
${\bf P}^f_{n}(z)$). In fact, it is also sufficient due the 
following theorem established in \cite{bkjfa}.
\begin{theorem}
Let $f\in\cS$, $t_0\in\T$ and $n\in{\mathbb N}$. 
The following are equivalent:
\begin{enumerate}
\item $f\in\cS^{(n)}(t_0)$.
\item The boundary Schwarz-Pick matrix ${\bf P}^f_{n}(t_0)$ exists.
\item The nontangential boundary limits $f_j(t_0)$ exist for
$j=0,\ldots,2n-1$ and satisfy
$$
|f_0(t_0)|=1\quad\mbox{and}\quad \PP^f_{n}(t_0)\ge 0,
$$
where $\PP^f_{n}(t_0)$ is the matrix defined in (\ref{1.8}).
\end{enumerate}
Moreover, if this is the case, then ${\bf P}^f_{n}(t_0)=\PP^f_{n}(t_0)$. 
\label{T:2.1}
\end{theorem}
We remark that in contrast to the boundary Schwarz-Pick matrix ${\bf 
P}^f_{n}(t_0)$ which is positive semidefinite whenever it exists, the 
structured matrix $\PP^f_{n}(t_0)$ defined in terms of the angular limits 
$f_j(t_0)$ by formula (\ref{1.8}) does not have to be positive 
semidefinite and even Hermitian. Theorem \ref{T:2.1} states in particular 
that positivity of this structured matrix is an exclusive property of 
$\cS^{(n)}(t_0)$-class functions. The following stronger version of the 
implication $(3)\Rightarrow (1)$ in Theorem \ref{T:2.1} appears in 
Theorem 1.7 \cite{bknach}.
\begin{theorem}
Let $f\in\cS$, $t_0\in\T$ and let us assume that the nontangential
boundary limits
$f_j(t_0)$ exist for $j=0,\ldots,2n-1$ and are such that
$\; |f_0(t_0)|=1\;$ and $\;\PP^f_{n}(t_0)=\PP^f_{n}(t_0)^*$.
Then $f\in\cS^{(n)}(t_0)$.
\label{T:2.2} 
\end{theorem}
\begin{remark}
{\rm Theorems \ref{T:2.1} and \ref{T:2.2} show that for $f\in\cS$ 
such that the boundary limits $f_j(t_0)$ exist for $j=0,\ldots,2n-1$
and $|f_0|=1$,
the matrix  $\PP^f_{n}(t_0)$ defined in (\ref{1.8}) is Hermitian if and
only if it is positive semidefinite and  moreover, that this is the case 
if and only if $f\in\cS^{(n)}(t_0)$}.
\label{R:2.3}
\end{remark}
In the rest of the section we prove the ``if'' part of statement (1) 
in Theorem \ref{T:1.2}. We first recall the following result
(see Theorem 6.2 in \cite{bkiwota} for the proof).
\begin{theorem}
Let $t_0\in\T$ and ${\bf s}=\{s_0,\ldots,s_{2n-1}\}$ be such that
\begin{equation}
|s_0|=1,\quad \PP^{\bf s}_n\ge 0\quad\mbox{and}\quad
\det \, \PP^{\bf s}_n=0.
\label{2.2}
\end{equation}
Then there exists a unique $f\in\cS$ such that 
\begin{equation}
f_j(t_0)=s_j \; \; (j=0,\ldots,2n-2)\quad\mbox{and}\quad
(-1)^nt_0^{2n-1}\overline{s}_0(f_{2n-1}(t_0)-s_{2n-1})\ge 0.
\label{2.3}
\end{equation}
This unique $f$ is a finite Blaschke
product of degree equal to the rank of $\PP^{\bf s}_{n}$.  
\label{T:2.4}
\end{theorem}
\begin{lemma}  
Let $g\in\cS^{(n)}(t_0)$. If $g$ is a finite Blaschke product, then 
\begin{equation}
{\rm rank} \, \PP^g_{n}(t_0)=\min\{n, \deg g\}.
\label{2.4}
\end{equation} 
Otherwise, $\PP^g_{n}(t_0)>0$,
\label{L:2.5}
\end{lemma}   
\begin{proof}
Since $g\in\cS^{(n)}(t_0)$, from Theorem \ref{T:2.1} we have
$|g_0(t_0)|=1$ and $\PP^g_{n}(t_0)\ge 0$. Let us assume that
$\PP^g_{n}(t_0)$ is singular and that ${\rm rank} \, (\PP^g_{n}(t_0))=d$.
Letting $s_j:=g_j(t_0)$ for $j=0,\ldots,2n-1$, we conclude from 
Theorem \ref{T:2.4} that there exists a unique function $f\in\cS$
satisfying conditions (\ref{2.3}) and that $f$ is a  Blaschke product 
of ${\rm deg} f=d$. Since $g$ obviously satisfies the same conditions, 
we have $f\equiv g$. Thus, if $\PP^g_{n}(t_0)\ge 0$ is singular, then 
$g$ is a finite Blaschke product and ${\rm rank} \, \PP^g_{n}(t_0)=
{\rm deg} g<n$. To complete the proof it remains to show that if $g$ is 
a finite Blaschke product and ${\rm rank} \, \PP^g_{n}(t_0)=n$, then 
$\deg g\ge n$. To this end, observe that since $\PP^g_{n}(t_0)=\lim_{z\to 
t_0}\PP^g_{n}(z)$ and since ${\rm rank} \, \PP^g_{n}(z)=\min\{n, \deg 
g\}$ for every $z\in\D$, we have
$$
n={\rm rank} \, \PP^g_{n}(t_0)\le {\rm rank} \, \PP^g_{n}(z)=
\min\{n, \deg g\}.
$$
Therefore, $\deg g\ge n$ which completes the proof.\end{proof}
\begin{corollary}
Let $N\ge 2n+1$, let $t_0\in\T$ and ${\bf s}=\{s_0,\dots,s_N\}$ be such 
that (\ref{2.2}) holds and let us assume that $\PP^{\bf s}_{n+1}\not\ge 
0$. Then  the problem ${\bf BP}_{N}$ has no 
solutions.
\label{C:2.5}
\end{corollary} 
\begin{proof}
 Assume that $f$ is a solution to the ${\bf
BP}_{N}$. Then $f$ satisfies conditions (\ref{2.3})
and therefore, it is a finite Blaschke product of degree $d={\rm rank} \,
\PP^{\bf s}_n<n$. Since $f$ solves the problem ${\bf
BP}_{N}$ and since $N\ge 2n+1$, it follows that $f_{2n}(t_0)=s_{2n}$ and
$f_{2n+1}(t_0)=s_{2n+1}$. Therefore $\PP^f_{n+1}(t_0)=\PP^{\bf
s}_{n+1}$. Since $f\in\cS$, the matrix $\PP^f_{n+1}(t_0)$
is positive semidefinite, and so is $\PP^{\bf s}_{n+1}$, which 
contradicts the assumption. 
\end{proof}
\begin{corollary}
Let $N=2n-1$ or $N=2n$ and let $t_0\in\T$ and ${\bf s}=\{s_0,\dots,s_N\}$ 
be 
such that (\ref{2.2}) holds. Then  the problem ${\bf BP}_{N}$ has a
(unique) solution if and only if the matrix $\PP^{\bf s}_{n}$ admits 
a positive semidefinite structured extension $\PP^{\bf s}_{n+1}$.
\label{C:2.6}
\end{corollary}
\begin{proof}
Uniqueness follows from Theorem \ref{T:2.4}.
If $f$ solves the ${\bf BP}_{N}$, then it is a finite 
Blaschke product (by Theorem \ref{T:2.4}) and therefore $f_j(t_0)$ exist 
for every $j\ge 0$. Letting $s_{2n+1}:=f_{2n+1}(t_0)$ and also 
$s_{2n}:=f_{2n}(t_0)$ (in case $N=2n-1$ where $s_{2n}$ is not prescribed)
we have $\PP^{\bf s}_{n+1}=\PP^f_{n+1}(t_0)\ge 0$ which proves the 
``only if'' part. Conversely, if $\PP^{\bf s}_{n+1}\ge 0$ for some 
choice of $s_{2n}$ and $s_{2n+1}$ (in case $N=2n-1$) or for some choice
of $s_{2n+1}$ (if $N=2n$ and hence $s_{2n}$ is prescribed), then we 
conclude by virtue of Theorem \ref{T:2.4} that there is an $f\in\cS$ such 
that
$$
f_j(t_0)=s_j \; \;
(j=0,\ldots,2n)\quad\mbox{and}\quad
(-1)^nt_0^{2n+1}\overline{s}_0(f_{2n+1}(t_0)-s_{2n+1})\ge 0.
$$
This  $f$ clearly is a solution to the problem ${\bf BP}_{N}$ for 
either $N=2n-1$ or $N=2n$.
\end{proof}
\begin{lemma}
Let us assume that  $t_0\in\T$ and ${\bf s}=\{s_0,\ldots,s_{2n-1}\}$ meet 
conditions (\ref{2.2}). Then the problem ${\bf 
BP}_{2n-1}$ has a (unique) solution if and only if $\; \; {\rm rank} \, 
\PP^{\bf s}_n={\rm rank} \, \PP^{\bf s}_{n-1}$.
\label{L:2.6}
\end{lemma}
\begin{proof}
 By Theorem \ref{T:2.4}, there exists a unique $f\in\cS$ 
satisfying conditions (\ref{2.3}), which is a finite Blaschke product of 
degree $d={\rm rank} \, \PP^{\bf s}_n<n$. This $f$ may or may not be a 
solution of the problem ${\bf BP}_{2n-1}$, i.e., it does or does not
satisfy equality  $f_{2n-1}(t_0)=s_{2n-1}$ rather than inequality in
(\ref{2.3}). If it does, then $\PP^f_{n}(t_0)=\PP^{\bf s}_n$ and 
therefore, we have from (\ref{2.4})
$$
{\rm rank} \, \PP^{\bf s}_{n-1}={\rm rank} \, (\PP^f_{n-1}(t_0))=
\min\{n-1, d\}=d={\rm rank} \, \PP^{\bf s}_{n}
$$
which proves the ``only if'' part. To verify the reverse direction, 
let us assume that the only function $f$ satisfying conditions 
(\ref{2.3}) is not a solution to the problem ${\bf BP}_{2n-1}$, i.e., that 
the  strict inequality prevails in (\ref{2.3}). Then it follows from the 
definitions (\ref{1.7}) and (\ref{1.8}) that 
all the corresponding entries in $\PP^{f}_{n}(t_0)$ and $\PP^{\bf s}_{n}$
are equal, except for the rightmost diagonal entries $p^f_{nn}$ and
$p^{\bf s}_{nn}$ which are subject to $p^f_{nn}<p^{\bf s}_{nn}$. Write 
$\PP^{\bf s}_{n}$ and $\PP^{f}_{n}(t_0)$ in the block form as
$$
\PP^{\bf s}_{n}=\left[\begin{array}{cc}\PP^{\bf s}_{n-1} &
B \\ B^* & p^{\bf s}_{nn}\end{array}\right],\quad
\PP^{f}_{n}(t_0)=\left[\begin{array}{cc}\PP^{\bf   
s}_{n-1} & B \\ B^* & {p}^f_{nn}\end{array}\right].
$$
Since the latter matrices are positive semidefinite, we have by the 
standard Schur complement argument,
\begin{eqnarray}
{\rm rank} \,  \PP^{\bf s}_{n}&={\rm rank} \,  \PP^{\bf s}_{n-1}+{\rm
rank} \, (p^{\bf s}_{nn}-X^*\PP^{\bf s}_{n-1}X),\label{2.5}\\
{\rm rank} \, \PP^{f}_{n}(t_0)&={\rm rank} \, \PP^{\bf s}_{n-1}+{\rm rank}
\, (p^f_{nn}-X^*\PP^{\bf s}_{n-1}X),\label{2.6}
\end{eqnarray}
where $X\in\C^{n-1}$ is any solution of the equation  $\PP^{\bf
s}_{n-1}X=B$. Since ${\rm rank} \, \PP^{f}_{n}(t_0)={\rm rank} \,
\PP^{f}_{n-1}(t_0)={\rm rank} \, \PP^{\bf s}_{n-1}$, it follows from   
(\ref{2.6}) that $p^f_{nn}=X^*\PP^{\bf s}_{n-1}X$. Since
$p^f_{nn}<p^{\bf s}_{nn}$, we conclude from (\ref{2.5}) that
$$
{\rm rank} \, \PP^{\bf s}_{n}={\rm rank} \, \PP^{\bf s}_{n-1}+1.
$$
Thus, ${\rm rank} \, \PP^{\bf s}_{n}\neq {\rm rank} \, \PP^{\bf s}_{n-1}$
which completes the proof.\end{proof}

To proceed, we need the following ``symmetry'' result.
\begin{lemma}
Let us assume that $t_0\in\T$ and ${\bf s}=\{s_0,\dots,s_{2n-1}\}$ are such that 
\begin{equation}
|s(t_0)|=1\quad\mbox{and}\quad \PP^{\bf s}_{n}=\PP^{\bf s *}_{n}.
\label{dop}
\end{equation}
Let $p_{ij}^{\bf s}$ be the numbers defined via formula (\ref{1.9}) for 
\begin{equation}
i,j\in\{1,\ldots,2n-2\},\quad\mbox{subject to}\quad 2\le i+j\le 2n-2.
\label{dop1}
\end{equation}
Then $p_{ij}^{\bf s}=\overline{p}_{ji}^{\bf s}$ for all $i,j$ as in (\ref{dop1}).
\label{L:6}
\end{lemma}
Observe that the positive definitness of the associated matrix $\PP^{\bf s}_{n}$
is not required. Note also that since the numbers $\Psi_{j\ell}$ in \eqref{1.6} are 
defined for all $j,\ell\ge 1$, the data set $\{t_0,s_0,s_1,\ldots,s_{2n-1}\}$ is 
exactly what we need to define the numbers $p_{ij}^{\bf s}$ for the indeces $(i,j)$ as 
in (\ref{dop1}). The statement follows by combining some results from \cite{bknach}
and \cite{boldym1}. We will give the exact references below.
\begin{proof}
By \cite[Theorem 1.9]{bknach}, conditions (\ref{dop}) are equivalent to the following 
matrix equality 
\begin{equation}
{\mathbb U}^{\bf s}_{2n}{\bf\Psi}_{2n}(t_0)\overline{\mathbb U}^{\bf s}_{2n}=
{\bf\Psi}_{2n}(t_0),
\label{dop2}
\end{equation}
where the $2n\times 2n$ upper triangular matrices ${\mathbb U}^{\bf s}_{2n}$ and 
${\bf\Psi}_{2n}(t_0)$ are defined via formulas (\ref{1.5}) and (\ref{1.55}) and where
$\overline{\mathbb U}^{\bf s}_{2n}$ denotes the complex conjugate of ${\mathbb U}^{\bf s}_{2n}$.
Let  us define the matrices $T_{2n}\in\C^{2n\times 2n}$ and $E_{2n}, \,  M_{2n}\in\C^{2n}$ by  
the formulas
\begin{equation}
T_{2n}=\left[\begin{array}{cccc} t_0 & 0 & \ldots & 0 \\
1 & t_0 & \ddots & \vdots \\
& \ddots & \ddots & 0 \\
0 && 1 & t_0\end{array}\right],\quad E_{2n}=\left[\begin{array}{c}1 \\ 0 \\
\vdots \\ 0\end{array}\right],\quad M_{2n}=
\left[\begin{array}{c}s_0 \\ s_1 \\ \vdots \\ s_{2n-1}
\end{array}\right]
\label{dop3}
\end{equation}
By \cite[Theorem 10.5]{boldym1}, condition (\ref{dop2} is necessary and sufficient for
the Stein equation 
\begin{equation}
Q-T_{2n}QT_{2n}^*=E_{2n}E_{2n}^*-M_{2n}M_{2n}^*.
\label{dop4}
\end{equation}
to have a solution $Q=\left[q_{ij}\right]_{i,j=1}^{2n}$. It is not hard to see 
(see \cite[Lemma 11.1]{boldym1} that the entries $q_{ij}$ are uniquely recovered from
\eqref{dop4} for all $(i,j)$ as in \eqref{dop1}; the explicit formula for each such $q_{ij}$
coincides with that in (\ref{1.9}) for the corresponding $p^{\bf s}_{ij}$. Thus, $q_{ij}=p^{\bf 
s}_{ij}$ for all $(i,j)$ subject to (\ref{dop1}). On the other hand, by taking adjoints
in \eqref{dop4} we conclude that $Q^*$ solves \eqref{dop4} whenever $Q$ does. By the 
above uniqueness, the $(i,j)$-th entry of $Q^*$ (which is $\overline{q}_{ji}$ equals 
$p^{\bf s}_{ij}$ for every $(i,j)$ as in \eqref{dop1}. Therefore, $p^{\bf 
s}_{ij}=\overline{q}_{ji}=\overline{p}^{\bf s}_{ji}$ for every $(i,j)$ subject to \eqref{dop1},
which completes the proof.\end{proof}
\begin{lemma} 
Let us assume that $t_0\in\T$ and ${\bf 
s}=\{s_0,\ldots,s_{2n-1},s_{2n}\}$ meet conditions (\ref{2.2}).
Then the problem ${\bf BP}_{2n}$ has a (unique) solution if and only if
(\ref{1.11}) hold.
\label{L:2.9}
\end{lemma}
\begin{proof} Due to Corollary \ref{C:2.6}, it suffices to show that
if conditions (\ref{2.2}) are satisfied, then conditions (\ref{1.11}) are 
necessary and sufficient for the existence of an $s_{2n+1}\in\C$ such that 
the matrix $\PP^{\bf s}_{n+1}$ defined via formula (\ref{1.7}) is 
positive semidefinite. Write $\PP^{\bf s}_{n+1}$ in the form 
$$
\PP^{\bf s}_{n+1}=\left[\begin{array}{cc}\PP^{\bf s}_{n} &
B_n \\ C_n & p^{\bf s}_{n+1,n+1}\end{array}\right],\quad\mbox{where}\quad
C_n=[p^{\bf s}_{n+1,1} \; \; p^{\bf s}_{n+1,2} \; \ldots \, 
p^{\bf s}_{n+1,n}],
$$
where $B_n$ is given in (\ref{1.10}) and where accordingly to 
(\ref{1.9}),
\begin{align}
p^{\bf s}_{n+1,n+1}=&\sum_{r=1}^{n-1}\sum_{\ell=1}^{r}
s_{n+\ell}\Psi_{\ell r}\overline{s}_{n+1-r}\notag\\
&+\sum_{\ell=1}^n 
s_{n+\ell}\Psi_{\ell,n+1}\overline{s}_{n+1-r}
+(-1)^{n}t_0^{2n+1}s_{2n+1}\overline{s}_0.
\label{2.7}
\end{align}
Recall that the entry $p^{\bf s}_{n+1,n+1}$ in $\PP^{\bf s}_{n+1}$ is the 
only one which depends on $s_{2n+1}$. Formula (\ref{2.7})
shows that one can get any $p^{\bf s}_{n+1,n+1}\in\C$ by an appropriate 
choice of $s_{2n+1}$. Since $\PP^{\bf s}_{n}$ is Hermitian and $|s_0|=1$, 
it follows from Lemma \ref{L:6} that $p^{\bf s}_{ij}=\overline{p}^{\bf s}_{ji}$ 
for every $(i,j)$ subject  to (\ref{dop1}; in particular, $p^{\bf 
s}_{n+1,j}=\overline{p}^{\bf s}_{j,n+1}$ for 
every $j=1,\ldots,n-2$. Therefore, the first condition in (\ref{1.11}) is 
equivalent to $C_n=B_n^*$ so that 
\begin{equation}
\PP^{\bf s}_{n+1}=\left[\begin{array}{cc}\PP^{\bf s}_{n} &
B_n \\ B_n^* & p^{\bf s}_{n+1,n+1}\end{array}\right],
\label{2.8}
\end{equation}
where $B_n$ is given in (\ref{2.2}). A well known result on positive 
semidefinite block matrices asserts that the matrix (\ref{2.8}) is 
positive semidefinite if and only if the equation 
\begin{equation}
\PP^{\bf s}_{n}X=B_n
\label{2.9}
\end{equation}
is consistent and $p^{\bf s}_{n+1,n+1}\ge X^*\PP^{\bf s}_{n}X$ 
for any solution $X$ to (\ref{2.9}). Thus, the matrix (\ref{2.8})
is positive semidefinite {\em for some $p^{\bf s}_{n+1,n+1}$} (or 
equivalently, for some $s_{2n+1}$) if and only if equation (\ref{2.9}) 
is consistent. The latter is equivalent to the second condition in 
(\ref{1.11}).
\end{proof}

\section{The indeterminate case}
\setcounter{equation}{0}

In this section we consider the cases listed in the second 
part of Theorem \ref{T:1.2}. Since the case where $|s_0|<1$ is covered by 
Lemma \ref{L:1.1}, we can (and will) assume that $t_0\in\T$ and 
${\bf s}=\{s_0,\ldots,s_{N}\}$ are such that 
\begin{equation}
|s_0|=1\quad\mbox{and}\quad \PP^{\bf s}_n> 0
\label{3.1}   
\end{equation}
where $\PP^{\bf s}_n$ is defined by formulas
(\ref{1.5})--(\ref{1.7}). For the maximal case where $N=2n-1$, the 
complete parametrization of all solutions of the {\bf BP}$_{2n-1}$
is known and will be recalled in Theorem \ref{T:3.2} below. 
Let  $T\in\C^{n\times n}$ and $E, \,  M\in\C^n$ be the matrices given  by
\begin{equation}
T=\left[\begin{array}{cccc} t_0 & 0 & \ldots & 0 \\
1 & t_0 & \ddots & \vdots \\
& \ddots & \ddots & 0 \\
0 && 1 & t_0\end{array}\right],\quad E=\left[\begin{array}{c}1 \\ 0 \\ 
\vdots \\ 0\end{array}\right],\quad M=
\left[\begin{array}{c}s_0 \\ s_1 \\ \vdots \\ s_{n-1}
\end{array}\right]
\label{3.2}
\end{equation}
(these matrices are of the same structure as those in (\ref{dop3}) but 
twice smaller) and let $\widetilde{\PP}$ be the positive definite matrix defined as
\begin{equation}
\widetilde{\PP}:=\PP^{\bf s}_n+MM^*.
\label{3.3}
\end{equation}
It is not hard to show that the numbers  
$M^*\widetilde{\PP}^{-1}M$ and $E^*\widetilde{\PP}^{-1}E$ are less than 
one. We let
$$
\alpha=\sqrt{1-M^*\widetilde{\PP}^{-1}M},\quad\mbox{and}\quad
\beta=\sqrt{1-E^*\widetilde{\PP}^{-1}E}.
$$
Now we introduce the $2\times 2$ matrix-function
\begin{equation}
{\bf S}=\left[\begin{array}{cc}{\bf a} &
{\bf b} \\ {\bf c} & {\bf d}\end{array}\right]
\label{3.5}
\end{equation}
 with the entries 
\begin{eqnarray}
{\bf a}(z)&=& E^*(\widetilde{\PP}-z \PP^{\bf s}_n T^*)^{-1}M,\label{3.6}\\
{\bf b}(z)&=&\beta\left(1- z E^*(\widetilde{\PP}-z \PP^{\bf s}_n
T^*)^{-1}T^{-1}E\right),\label{3.7}\\
{\bf c}(z)&=&\alpha\left(1- z M^*T^*(\widetilde{\PP}-z \PP^{\bf s}_n
T^*)^{-1}M\right),\label{3.8}\\
{\bf d}(z)&=&z\alpha\beta M^* 
(\PP^{\bf s}_n)^{-1}\widetilde{\PP}(\widetilde{\PP}-z
\PP^{\bf s}_n T^*)^{-1}T^{-1}E.\label{3.9}
\end{eqnarray}
It was shown in Theorem 6.4 \cite{bkiwota} that ${\bf S}$ is a 
rational function of McMillan degree $n$ which is inner in $\D$. 
Therefore, its entries (\ref{3.6})--(\ref{3.9})
are rational Schur class functions  analytic at $t_0$. Some 
properties of their Taylor coefficients at $t_0$ are recalled below
(see Lemma 6.5 in \cite{bkiwota} for the proof).
\begin{theorem}
Let 
\begin{equation}
{\bf a}(z)=\sum_{j\ge 0}a_j(t_0)(z-t_0)^j,\quad {\bf 
b}(z)={\displaystyle\sum_{j\ge 0}b_j(t_0)(z-t_0)^j},\quad  
{\bf c}(z)={\displaystyle\sum_{j\ge 0}c_j(t_0)(z-t_0)^j}
\label{3.10}
\end{equation}
be the Taylor 
expansions of the functions (\ref{3.6})--(\ref{3.8}) at $t_0$. Then
\begin{enumerate}
\item $a_j(t_0)=s_j\; $ for $\; j=0,\ldots,2n-1\; $ and $\; 
|{\bf d}(t_0)|=1$.
\item $b_j(t_0)=c_j(t_0)=0\; $ for $\; j=0,\ldots,n-1$.
\item $b_{n}(t_0)\neq 0$, $\; c_{n}(t_0)\neq 0\;$ and moreover,
\begin{equation}
{t}_0^{2n}b_{n}(t_0)=(-1)^{n-1}\overline{c_{n}(t_0)}{\bf d}(t_0)s_0.
\label{3.11}
\end{equation}
\end{enumerate}
\label{T:3.1}
\end{theorem}
The next theorem (see Theorem 1.6 in \cite{bkiwota}) 
describes the solution set of the problem ${\bf BP}_{2n-1}$, that is,
all functions $f\in\cS$ such that  
\begin{equation}
f(z)=s_0+s_1(z-t_0)+\ldots+s_{2n-1}(z-t_0)^{2n-1}+o(|z-t_0|^{2n-1})
\label{3.17}
\end{equation}
and also the solution set of its slight modification
$\widetilde{\bf BP}_{2n-1}$ which consists of finding 
$f\in\cS$ subject to the stronger nontangential asymptotic
\begin{equation}
f(z)=s_0+s_1(z-t_0)+\ldots+s_{2n-1}(z-t_0)^{2n-1}+O(|z-t_0|^{2n})
\quad\mbox{at} \; \; t_0.
\label{3.12}
\end{equation}
\begin{theorem}   
Let us assume that conditions (\ref{3.1}) are in force.
\begin{enumerate}
\item A function $f$ is a solution to the problem $\widetilde{\bf 
BP}_{2n-1}$
if and only if it is of the form
\begin{equation}
f(z)={\bf T}_{\bf S}[\cE](z):={\bf a}(z)+\frac{{\bf b}(z){\bf 
c}(z)\cE(z)}{1-{\bf d}(z)\cE(z)}
\label{3.13}  
\end{equation}
where the coefficient matrix ${\bf S}$ is given in 
(\ref{3.5})--(\ref{3.9}) and where $\cE$ is a Schur-class 
function such that either 
\begin{equation}
\cE(t_0):=\lim_{z\too t_0}\cE(z)\neq  \overline{{\bf d}(t_0)}
\label{3.14}  
\end{equation}
or the nontangential boundary limit $\cE(t_0)$ does not exist.
\item A function $f$ solves the problem ${\bf BP}_{2n-1}$, i.e., 
$f\in\cS$ and satisfies (\ref{3.17})
if and only if $f$ is of the form (\ref{3.13}) for 
an $\cE\in\cS$ which is either as in $(1)$ or is 
subject to equalities
\begin{equation}
\cE(t_0)=\overline{{\bf d}(t_0)}\quad\mbox{and}\quad \liminf_{z\to t_0} 
\frac{1-|\cE(z)|^2}{1-|z|^2}=\infty.
\label{3.15}
\end{equation}
\end{enumerate}
\label{T:3.2}
\end{theorem}
\begin{remark}
{\rm The correspondence $\cE\to f$ established by formula (\ref{3.13}) is 
one-to-one and the inverse transformation is given by 
\begin{equation}
\cE(z)={\bf T}^{-1}_{\bf S}[f](z)=
\frac{f(z)-{\bf a}(z)}{{\bf b}(z){\bf c}(z)+{\bf d}(z)(f(z)-{\bf 
a}(z))}.
\label{3.16}
\end{equation}  
Therefore condition (\ref{3.15})  explicitly describes the 
dichotomy between condition (\ref{3.12}) and a weaker condition 
(\ref{3.17}).  Although condition (\ref{3.12}) does not have a clear 
interpolation interpretation in general, it gets one while being 
restricted 
to {\em rational} Schur functions. In this case, (\ref{3.12}) is 
equivalent to (\ref{3.17}) and therefore, to conditions (\ref{1.3}); we
refer to \cite{BGR} for rational boundary interpolation.} 
\label{R:3.3}
\end{remark}
\begin{remark}
{\rm Substituting all Schur class functions $\cE$ into (\ref{3.13})
produces all functions $f\in\cS$ satisfying conditions 
(\ref{2.3}). This relaxed interpolation problem was 
studied in \cite{Kov}, \cite{boldym1}, \cite{bkiwota}.
Theorem \ref{T:3.2} also describes the gap between 
the problem  ${\bf BP}_{2n-1}$ and its relaxed version: the strict 
inequality  holds in the last condition in (\ref{2.3})
for a function $f$ of the form (\ref{3.13}) if and
only if the corresponding parameter $\cE$ is subject to
$\cE(t_0)=\overline{{\bf d}(t_0)}$ and ${\displaystyle\liminf_{z\to 
t_0}\frac{1-|\cE(z)|^2}{1-|z|^2}}<\infty$.}
\label{R:3.4}
\end{remark}
Theorem \ref{T:3.2} shows that conditions (\ref{3.1}) guarantee that
the problem  {\bf BP}$_{2n-1}$ has infinitely many solutions which 
covers therefore the case (b1) in the second part of Theorem 
\ref{T:1.2}. In case $N\ge 2n$ we will use representation (\ref{3.13})
to reduce the original problem {\bf BP}$_N$ to a similar problem with 
fewer number of interpolation conditions. Still assuming that conditions 
(\ref{3.1}) are satisfied we use Taylor expansions (\ref{3.10}) of 
rational functions  (\ref{3.6})--(\ref{3.8}) at $t_0$ and the Taylor expansion
${\bf d}(z)={\displaystyle\sum_{j=0}^\infty d_j(t_0)(z-t_0)^j}$ of the function ${\bf d}$ from
(\ref{3.9}) to define the polynomials
\begin{equation}
F(z)=\sum_{j=0}^{N-2n}(s_{2n+j}-a_{2n+j}(t_0))(z-t_0)^j,\quad
D(z)=\sum_{j=0}^{N-2n}d_{j}(t_0)(z-t_0)^j,
\label{3.18}
\end{equation}
\begin{equation}
B(z)=\sum_{j=0}^{N-2n}b_{n+j}(t_0)(z-t_0)^j ,\quad 
C(z)=\sum_{j=0}^{N-2n}c_{n+j}(t_0)(z-t_0)^j
\label{3.19} 
\end{equation}
and the rational function
\begin{equation}
R(z)=\frac{F(z)}{B(z)C(z)+D(z)F(z)}.
\label{3.20}  
\end{equation}
Observe that since $B(t_0)C(t_0)=b_n(t_0)c_n(t_0)\neq 0$ (by part (3) in 
Theorem \ref{T:3.1}), the numerator and the denominator in (\ref{3.20})
cannot have a common zero at $t_0$. Thus, $R(z)$ is analytic at $t_0$
if and only if 
\begin{equation}
B(t_0)C(t_0)+D(t_0)F(t_0)=b_n(t_0)c_n(t_0)+{\bf 
d}(t_0)(s_{2n}-a_{2n}(t_0))\neq 0.
\label{3.21}
\end{equation}
\begin{remark}
If condition (\ref{3.21}) is satisfied, then 
\begin{equation}
R_0:=R(t_0)=\frac{\overline{{\bf
d}(t_0)}(s_{2n}-a_{2n}(t_0))}{(-1)^{n-1}\overline{t}_0^{2n}|c_n(t_0)|^2s_0+
s_{2n}-a_{2n}(t_0)}\neq \overline{{\bf d}(t_0)}.
\label{3.22}  
\end{equation}
\label{R:3.5}
\end{remark}
\begin{proof}
Evaluating (\ref{3.20}) at $z=t_0$ gives, on account of 
(\ref{3.18}), (\ref{3.19}),
$$
R(t_0)=\frac{s_{2n}-a_{2n}(t_0)}{b_{n}(t_0)c_n(t_0)+{\bf
d}(t_0)(s_{2n}-a_{2n}(t_0))}
$$
and substituting (\ref{3.11}) into the right-hand side part of the 
latter equality gives
$$
R(t_0)=\frac{s_{2n}-a_{2n}(t_0)}{(-1)^{n-1}\bar{t}_0^{2n}{\bf
d}(t_0)s_0|c_n(t_0)|^2+{\bf d}(t_0)(s_{2n}-a_{2n}(t_0))}
$$
which is equivalent to the second equality in (\ref{3.22}) since $|{\bf 
d}(t_0)|=1$ (by part (1) in Theorem \ref{T:3.1}). Since 
$s_0\neq 0$ (by assumption (\ref{3.1}))
and $c_n(t_0)\neq 0$ (by part (3) in Theorem \ref{T:3.1}), the 
inequality in (\ref{3.22}) follows.
\end{proof}
\begin{theorem}
Let $t_0\in\T$ and ${\bf s}=\{s_0,\ldots,s_{N}\}$ be such that
conditions (\ref{3.1}) hold for some $n\le N/2$ and let ${\bf a}$,
${\bf b}$, ${\bf c}$, ${\bf d}$ be the rational functions defined
in (\ref{3.6})--(\ref{3.9}) with Taylor expansions (\ref{3.10})
at $t_0$. 
\begin{enumerate}
\item If the problem {\bf BP}$_{N}$ admits a solution, then (\ref{3.21})
holds, so that the function $R$ defined in (\ref{3.20}) is analytic at 
$t_0$.
\item If condition (\ref{3.21}) is satisfied, then a function  $f$
is a solution of the problem {\bf BP}$_{N}$ if and only if it is of the 
form (\ref{3.13}) for some $\cE\in\cS$ such that
\begin{equation}
\cE(z)=R(z)+o(|z-t_0|^{N-2n})\quad\mbox{as} \quad z\too t_0.
\label{3.23}  
\end{equation}
\end{enumerate}
\label{T:3.6}
\end{theorem}
\begin{proof}
By statement (1) in Theorem \ref{T:3.1}, $a_j(t_0)=s_j$
for $j=0,\ldots,2n-1$ which together with definition (\ref{3.18})
of $F$ implies that
$$
{\bf a}(z)+(z-t_0)^{2n}F(z)=\sum_{j=0}^Ns_j(z-t_0)^j+O(|z-t_0|^{N+1}).
$$
Therefore, asymptotic equality (\ref{1.2}) can be equivalently written as 
\begin{equation}
f(z)={\bf a}(z)+(z-t_0)^{2n}F(z)+o(|z-t_0|^{N})\quad (z\too
t_0).\label{3.24}    
\end{equation}  
Let $f$ be a solution to the {\bf BP}$_{N}$, i.e., $f\in\cS$ 
and (\ref{3.24}) holds. Since $N\ge 2n$, $f$ also satisfies (\ref{3.12}) 
and therefore it is of the form (\ref{3.13}) for some $\cE\in\cS$ (by 
Theorem \ref{T:3.2}). Observe the equalities
\begin{eqnarray}
{\bf d}(z)&=&D(z)+o(|z-t_0|)^{N-2n}\quad (z\to t_0),\label{3.24a}\\
{\bf b}(z){\bf c}(z)&=&(z-t_0)^{2n}B(z)C(z)+o(|z-t_0|^{N})\quad 
(z\to t_0)\label{3.25}
\end{eqnarray}
which follow from definitions (\ref{3.18}), (\ref{3.19}) by statement
(2) in Theorem \ref{T:3.1}. Substituting (\ref{3.24})--(\ref{3.25}) into 
(\ref{3.16}) (which is equivalent to (\ref{3.13})) gives
\begin{equation}
\cE(z)=\frac{F(z)+o(|z-t_0|^{N-2n})}{B(z)C(z)+D(z)F(z)+o(|z-t_0|^{N-2n})}
\quad (z\too t_0).
\label{3.26}
\end{equation}
Since $F$, $B$, $C$, $D$ are polynomials, the limit (as $z\too t_0$)
of the expression on the right hand side of (\ref{3.26}) exists
(finite of infinite) and therefore the limit $\cE(t_0)$ exists too.
Since $\cE$ is a Schur-class function, this limit is finite and therefore,
(\ref{3.21}) holds. Asymptotic equality (\ref{3.23}) follows from 
(\ref{3.20}) and (\ref{3.26}) due to (\ref{3.21}).

\smallskip

It remains to prove the ``if'' part in statement (2) of the theorem. 
To this end, let us assume that condition (\ref{3.21}) is met
so that $R$ is analytic at $t_0$. Let us assume that $\cE$ is a 
Schur-class function subject to asymptotic equality (\ref{3.23}) and let 
$f$ be defined by the formula (\ref{3.13}). Then $f\in\cS$ since 
$\cE\in\cS$ 
and the coefficient matrix (\ref{3.5}) is inner. Substituting 
(\ref{3.23}), (\ref{3.24a}) 
and (\ref{3.25}) into (\ref{3.13}) we obtain
\begin{eqnarray}
f(z)&=&{\bf a}(z)+\frac{[(z-t_0)^{2n}B(z)C(z)+o(|z-t_0|^N)]\cdot[
R(z)+o(|z-t_0|^{N-2n})]}{1-[D(z)+o(|z-t_0|^{N-2n})]\cdot
[R(z)+o(|z-t_0|^{N-2n})]}\nonumber\\
&=&{\bf a}(z)+\frac{(z-t_0)^{2n}B(z)C(z)R(z)+o(|z-t_0|^N)
}{1-D(z)R(z)+o(|z-t_0|^{N-2n})}.
\label{3.27}
\end{eqnarray}
By Remark \ref{R:3.5}, $R(t_0)\neq\overline{{\bf d}(t_0)}$ and 
since $|{\bf d}(t_0)|=1$ (by part (1) in Theorem \ref{T:3.1}),
it follows that $1-R(t_0)D(t_0)=1-R(t_0){\bf d}(t_0)\neq 0$.
Then we can write (\ref{3.27}) as 
$$
f(z)={\bf a}(z)+\frac{(z-t_0)^{2n}B(z)C(z)R(z)
}{1-D(z)R(z)}+o(|z-t_0|^N).
$$
Now we substitute formula (\ref{3.20}) for $R$ into the latter 
equality and arrive at (\ref{3.24}) which is equivalent to (\ref{1.2}).
Thus, $f$ solves {\bf BP}$_{N}$ which completes the proof of 
the theorem.
\end{proof}

\begin{corollary}
Let $t_0\in\T$ and ${\bf s}=\{s_0,\ldots,s_{N}\}$ meet
conditions (\ref{3.1}) for some $n\le N/2$ and let $f$ be of the 
form (\ref{3.13}) for some function $\cE\in\cS$ subject to (\ref{3.14}). 
Then the boundary limit $f_{2n}(t_0)$ exists if and only if the limit 
$\cE(t_0)$ exists. In this case, 
\begin{equation}
f_{2n}(t_0)=a_{2n}(t_0)+
\frac{(-1)^{n-1}\overline{t}_0^{2n}|c_n(t_0)|^2s_0\cE(t_0)}{\overline{{\bf
d}(t_0)}-\cE(t_0)}.
\label{3.27a}
\end{equation}
\label{C:3.6}
\end{corollary}
\begin{proof} Since conditions (\ref{3.1}) are met, representation (\ref{3.13})
for $f$ follows from Theorem \ref{T:3.2} (part (2)).
Simultaneous existence of the limits follows from Theorem 
\ref{T:3.6} (part (2)) applied to the problem {\bf BP}$_{2n}$ with data
$s_0,\ldots,s_{2n-1}$ and $s_{2n}:=f_{2n}(t_0)$.
Since $\cE(t_0)=R(t_0)$ by (\ref{3.23}), we have from (\ref{3.22})
$$
\cE(t_0)=\frac{\overline{{\bf
d}(t_0)}(f_{2n}(t_0)-a_{2n}(t_0))}{(-1)^{n-1}\overline{t}_0^{2n}|c_n(t_0)|^2s_0+
f_{2n}(t_0)-a_{2n}(t_0)}.
$$
Solving the latter equality for 
$f_{2n}(t_0)$ gives (\ref{3.27a}).\end{proof}
\begin{corollary}
The problem {\bf BP}$_N$ has a solution if and only if there exists 
a function $\cE\in\cS$ satisfying asymptotic equality  (\ref{3.23}) which 
in turn is equivalent to boundary interpolation conditions
\begin{equation}
\cE_j(t_0)=R_j(t_0)\quad\mbox{for}\quad j=0,\ldots,N-2n.
\label{3.27b}
\end{equation}
\label{C:3.7a} 
\end{corollary}
The first statement follows directly from part (2) of Theorem \ref{T:3.6}.
Since the function $R$ is analytic at $t_0$, the equivalence 
$(\ref{3.23})\Leftrightarrow(ref{3.27b})$ follows (see e.g., \cite[Crollary 7.9]{boldym1} 
for the proof). Explicit formula for $R_0=R(t_0)$ in terms of original data is given 
in (\ref{3.22}). Similar formulas for $j\ge 1$ can be written  explicitly 
but as we will see below, they do not play any essential role in the 
subsequent analysis.

\smallskip

Now we take another look at formula (\ref{3.22}). If we will think of 
$s_0,\ldots,s_{2n-1}$ as of given numbers satisfying conditions 
(\ref{3.1}), then formula (\ref{3.22}) establishes a linear fractional 
map
$F: \, s_{2n}\mapsto R_0$ on the Riemann sphere (recall that the entries 
${\bf d}(t_0)$, $c_n(t_0)$ and $a_{2n}(t_0)$ in (\ref{3.22}) are uniquely 
determined by $t_0$ and $s_0,\ldots,s_{2n-1}$). The only value
of the argument $s_{2n}$ which does not meet condition (\ref{3.21}) is
$s_{2n}^0=a_{2n}(t_0)-b_n(t_0)c_n(t_0)\overline{{\bf d}(t_0)}$. It is 
not hard to see from (\ref{3.21}) that $F(s_{2n}^0)=\infty$ and 
$F(\infty)=\overline{{\bf d}(t_0)}$. Thus, if we consider $F$ as a map 
from $\C\setminus\{s_{2n}^0\}$ into $\C$, then condition (\ref{3.21})
and inequality in (\ref{3.22}) will be satisfied automatically.

\smallskip

Still assuming that $t_0$, $s_0,\ldots,s_{2n-1}$ are fixed and varying 
$s_{2n}$, we can define two linear functions $s_{2n}\mapsto p^{\bf 
s}_{n+1,n}$ and 
$s_{2n}\mapsto p^{\bf s}_{n,n+1}$ by  formula (\ref{1.9}). Indeed,
letting $(i,j)=(n+1,n)$ and $(i,j)=(n,n+1)$ in (\ref{1.9})
and taking into account that 
$\Psi_{nn}=(-1)^{n-1}t_0^{2n-1}$ and $\Psi_{n+1,n+1}=(-1)^{n}t_0^{2n+1}$
(by (\ref{1.6})), we have
\begin{equation}
p^{\bf s}_{n+1,n}=(-1)^{n-1}t_0^{2n-1}s_{2n}\overline{s}_0+\Phi,\quad
p^{\bf s}_{n,n+1}=(-1)^{n}t_0^{2n+1}s_{2n}\overline{s}_0+\Upsilon
\label{3.29}
\end{equation}
where the terms
\begin{eqnarray}
\Phi&=&\sum_{r=1}^{n-1}\sum_{\ell=1}^rs_{n+\ell}\Psi_{\ell 
r}\overline{s}_{n-r}+\sum_{\ell=1}^{n-1}s_{n+\ell}\Psi_{\ell
n}\overline{s}_0,\label{3.30}\\
\Upsilon&=&\sum_{r=1}^{n}\sum_{\ell=1}^rs_{n+\ell-1}\Psi_{\ell
r}\overline{s}_{n+1-r}+\sum_{\ell=1}^{n}s_{n+\ell-1}\Psi_{\ell,
n+1}\overline{s}_0\label{3.31}
\end{eqnarray}
are completely determined from $t_0$ and $s_0,\ldots,s_{2n-1}$. 
\begin{lemma}
Let $R_0$, $p^{\bf s}_{n+1,n}$ and $p^{\bf s}_{n,n+1}$ be defined by 
formulas (\ref{3.22}) and (\ref{3.29}) for some fixed $s_{2n}$. 
Then
\begin{equation}
t_0(p^{\bf s}_{n+1,n}-\overline{p}^{\bf s}_{n,n+1})=
\frac{|c_n(t_0)|^2(1-|R_0|^2)}
{|\overline{\bf d}(t_0)-R_0|^2}.
\label{3.30a}
\end{equation}
\label{L:3.10}
\end{lemma}
\begin{proof} Let us substitute the constant function $\cE(z)\equiv
-{\overline{\bf d}(t_0)}$ into (\ref{3.13}):
$$
h(z):={\bf T}_{\bf S}[-\overline{{\bf d}(t_0)}](z)={\bf a}(z)-\frac{{\bf 
b}(z){\bf c}(z)\overline{{\bf d}(t_0)}}{1+{\bf d}(z)\overline{{\bf 
d}(t_0)}}.
$$
Since ${\mathcal E}$ is a unimodular
constant function and the matrix ${\bf S}$  of coefficients in
(\ref{3.13}) is inner, it follows that $h$ is a rational inner function,
i.e.,  a finite Blaschke product. Since $\cE(z)\equiv  
-{\overline{\bf d}(t_0)}$ meets condition (\ref{3.14}), the function $h$ 
solves the problem {\bf BP}$_{2n-1}$ by Theorem \ref{T:3.2}. Thus,
\begin{equation}
h_j(t_0)=s_j\quad\mbox{for}\quad j=0,\ldots,2n-1 
\label{3.32}  
\end{equation}
and therefore $\PP^{h}_n(t_0)=\PP^{\bf s}_n$ where the matrix 
$\PP^{h}_n(t_0)$ is defined via formula (\ref{1.8}). The extended 
matrix $\PP^{h}_{n+1}(t_0)$ is positive semidefinite, since $h$ is a 
finite Blaschke product. In particular, the $(n+1,n)$ and $(n,n+1)$
entries in this matrix are complex conjugates of each other: 
\begin{equation}
p^h_{n+1,n}=\overline{p}^h_{n,n+1}.
\label{3.33}
\end{equation}
These entries are defined via formula (\ref{1.9}) but with $h_j(t_0)$ 
replacing $s_j$. Due to (\ref{3.32}),
$$
p^h_{n+1,n}=(-1)^{n-1}t_0^{2n-1}h_{2n}(t_0)\overline{s}_0+\Phi,\quad
p^h_{n,n+1}=(-1)^{n}t_0^{2n+1}h_{2n}(t_0)\overline{s}_0+\Upsilon
$$
where $\Phi$ and $\Upsilon$ are the same as in (\ref{3.30}), 
(\ref{3.31}). Substituting the two latter equalities into (\ref{3.33}) we 
have after simple rearrangements,
\begin{equation}
\Phi-\overline{\Upsilon}=(-1)^{n}t_0^{2n-1}h_{2n}(t_0)\overline{s}_0
+(-1)^{n}\overline{t}_0^{2n+1}\overline{h_{2n}(t_0)}s_0
\label{3.34}
\end{equation}
The formula for $h_{2n}(t_0)$ can be obtained from Corollary \ref{C:3.6}
by plugging in $\cE(t_0)=-\overline{{\bf d}(t_0)}$ into (\ref{3.27a}):
$$
h_{2n}(t_0)=a_{2n}(t_0)+
\frac{(-1)^{n}}{2}\overline{t}_0^{2n}|c_n(t_0)|^2s_0.
$$
On the other hand, we have from (\ref{3.22})
$$
s_{2n}=a_{2n}(t_0)+
\frac{(-1)^{n-1}\overline{t}_0^{2n}|c_n(t_0)|^2s_0R_0}{\overline{{\bf
d}(t_0)}-R_0}
$$
and we conclude from the two last equalities that 
\begin{eqnarray}
t_0^{2n}\overline{s}_0(s_{2n}-h_{2n}(t_0))&=&(-1)^{n-1}|c_n(t_0)|^2
\left[\frac{R_0}{\overline{{\bf
d}(t_0)}-R_0}+\frac{1}{2}\right]\nonumber\\
&=&\frac{(-1)^{n-1}|c_n(t_0)|^2}{2}\cdot\frac{\overline{{\bf
d}(t_0)}+R_0}{\overline{{\bf d}(t_0)}-R_0}.
\label{3.35}
\end{eqnarray}
Now we make subsequent use of (\ref{3.29}), (\ref{3.34}) and (\ref{3.35})
to get
\begin{eqnarray*}
t_0(p^{\bf s}_{n+1,n}-\overline{p}^{\bf s}_{n,n+1})&=&
(-1)^{n-1}\left[t_0^{2n}s_{2n}\overline{s}_0+\overline{t}_0^{2n}
\overline{s}_{2n}s_0\right]+t_0\left[\Phi-\overline{\Upsilon}\right]\\
&=&(-1)^{n-1}\left[t_0^{2n}\overline{s}_0(s_{2n}-h_{2n}(t_0))+
\overline{t}_0^{2n}s_0(\overline{s}_{2n}-\overline{h_{2n}(t_0)})\right]\\
&=&|c_n(t_0)|^2\cdot {\rm Re} \left(\frac{\overline{{\bf
d}(t_0)}+R_0}{\overline{{\bf d}(t_0)}-R_0}\right)=
\frac{|c_n(t_0)|^2(1-|R_0|^2)}{|\overline{\bf d}(t_0)-R_0|^2}
\end{eqnarray*}
and thus, to complete the proof.
\end{proof}

\medskip

{\bf Proof of Theorem \ref{T:1.2}:} We will check all possible cases for 
given data $t_0$, $s_0,\ldots,s_N$. Recall that the integer $n$ is 
chosen so that the matrix $\PP^{\bf s}_n$ is positive semidefinite 
and the larger matrix $\PP^{\bf s}_{n+1}$ (in case $N>2n$) is not.

\medskip

{\bf Case 1:} If $|s_0|<1$, the problem has infinitely many solutions by 
Lemma \ref{L:1.1}.

\medskip

{\bf Case 2:} Let $|s_0|=1$ and $n=0$. Then the problem has no solutions.
Indeed, equality $n=0$ means (by the very definition of $n$) that 
$\PP^{\bf s}_1=t_0s_1\overline{s}_0\not\ge 0$. Then it follows from 
Theorem \ref{T:1.3a} that there are no Schur 
functions of the form (\ref{1.12}). Therefore, there are no Schur 
functions satisfying (\ref{1.2}), that is solving the problem {\bf BP}$_N$.

\medskip

{\bf Case 3:} Let $|s_0|=1$ and $\PP^{\bf s}_n$ is singular.  
By Corollary \ref{C:2.5}, Lemma \ref{L:2.6} and Lemma \ref{L:2.9}, the  
problem has a unique solution if $N=2n-1$ or $N=2n$ with additional 
conditions indicated in the formulation of part (1) in Theorem 
\ref{1.3}, and it does not have a solution otherwise.  

\medskip

{\bf Case 4:} If $|s_0|=1$, $\PP^{\bf s}_n>0$ and $N=2n-1$, then the 
problem has infinitely many solutions by Theorem \ref{T:3.2}.

\medskip

{\bf Case 5:} Let $N\ge 2n$,  $|s_0|=1$, $\PP^{\bf s}_n>0$, and 
$p^{\bf s}_{n+1,n}=\overline{p}^{\bf s}_{n,n+1}$. Then the problem has 
infinitely many solutions if $N=2n$ and it has no solutions if $N>2n$.

\smallskip

\begin{proof}
Let $N=2n$ so that $s_0,\ldots,s_{2n}$ are given and 
$s_{2n}$ is such that $p^{\bf s}_{n+1,n}=\overline{p}^{\bf s}_{n,n+1}$.
By the arguments from the proof of Lemma \ref{L:2.9}, 
there exists an $s_{2n+1}$ such that the structured extension
$\PP^{\bf s}_{n+1}$ of $\PP^{\bf s}_n$ is positive definite.
Since $\PP^{\bf s}_{n+1}>0$, it follows by virtue of 
Theorem 3.2 that there are infinitely many solutions to the problem {\bf 
BP}$_{2n+1}$ each one of which solves the {\bf BP}$_{2n}$.

\smallskip

To complete the proof we recall a result from \cite{bknach}
(see Theorem 1.8 there):

\medskip

{\em  Let $f\in\cS$ admit the nontangential boundary limits $f_j(t_0)$ 
for $j=0,\ldots,2n$ which are such that 
\begin{equation}
|f_0(t_0)|=1,\quad \PP^{f}_{n}(t_0)\ge 0\quad\mbox{and}\quad
p^f_{n+1,n}=\overline{p}^f_{n,n+1}.
\label{3.34a}
\end{equation}
If the nontangential boundary limit $f_{2n+1}(t_0)$ exists then 
necessarily $\PP^{f}_{n+1}\ge 0$}.
 
\medskip

Let $N>2n$ and let us assume that $f$ is a solution to the problem {\bf 
BP}$_{N}$. Since $N>2n$, we have enough data to construct $\PP^{\bf 
s}_{n+1}$ which must be equal to $\PP^{f}_{n+1}$. By the assumptions 
of the current case, conditions (\ref{3.34a}) are met and the limit 
$f_{2n+1}(t_0)$ exists. Therefore, the matrix $\PP^{f}_{n+1}=
\PP^{\bf s}_{n+1}$ is positive semidefinite which contradicts to the 
choice of $n$.\end{proof}

\medskip

{\bf Case 6:} Let $N\ge 2n$,  $|s_0|=1$, $\PP^{\bf s}_n>0$, and
$p^{\bf s}_{n+1,n}\neq \overline{p}^{\bf s}_{n,n+1}$. Then the problem has 
infinitely 
many solutions if $t_0\left(p^{\bf 
s}_{n+1,n}-\overline{p}^{\bf s}_{n,n+1}\right)>0$
and it has no solutions if  
$t_0\left(p^{\bf s}_{n+1,n}-\overline{p}^{\bf s}_{n,n+1}\right)<0$.

\medskip

\begin{proof} By Corollary \ref{C:3.7a},
the problem {\bf BP}$_{N}$ has a solution if and only if there is 
a function $\cE\in\cS$ satisfying conditions (\ref{3.27b}) with 
$R_0(t_0)$ given by (\ref{3.22}). By (\ref{3.30a}), the number 
$u:=t_0\left(p^{\bf s}_{n+1,n}-\overline{p}_{n,n+1}\right)$ is real.
Since $|c_n(t_0)|\neq 0$ (by statement (3) in Theorem \ref{T:3.1})
and $R_0\neq \overline{{\bf d}(t_0)}$ (by (\ref{3.22})), it follows from 
(\ref{3.30a}) that  $|R_0|<1$ if $u>0$ and $|R_0|<1$ if $u<0$. In the 
first case, there are infinitely many functions $\cE\in\cS$ satisfying
conditions (\ref{3.27b}) (by Lemma \ref{L:1.1}). Each such function 
lead to a solution $f$ of the problem {\bf BP}$_{N}$. In the second case 
there is no $\cE\in\cS$ satisfying $\cE(t_0)=R_0$ and therefore, there 
are no solutions to the {\bf BP}$_{N}$.\end{proof}

\medskip

All possible cases have been verified. They prove statement (3) and the 
``if'' parts in statements (1) and (2). Since these cases are disjoint, 
the  ``only if'' parts in statements (1) and (2) now follow. The fact 
that the unique solution (in part (1)) is a finite Blaschke product 
follows from Theorem \ref{T:2.4}. In the indeterminate case (2(b)), any 
solution 
of the problem belongs to $\cS^{(n)}(t_0)$, by Theorem \ref{T:2.2}.
This completes the proof of Theorem \ref{T:1.2}.

\end{document}